\documentclass[a4paper, 10pt]{article}
\usepackage[cp1250]{inputenc}
\usepackage[T1]{fontenc}
\usepackage{indentfirst}
\usepackage{amsthm}
\usepackage{amsmath}
\usepackage{amssymb}
\usepackage{amsfonts}
\usepackage[english]{babel}
\usepackage{a4wide}
\usepackage{fancyhdr}
\usepackage{tikz-cd}
\usepackage{enumerate}
\usepackage{titlesec}
\titlelabel{\thetitle.\quad}

\title{Factoring the Sobolev embedding operator}

\author{Krystian Kazaniecki, Aleksander Pawlewicz, Micha\l{} Wojciechowski}

\newtheorem{theorem}{Theorem}
\newtheorem*{theorem*}{Theorem}

\newtheorem{lemma}{Lemma}
\newtheorem{corollary}{Corollary}

\begin{document}
\maketitle

\begin{abstract}
The paper studies the factorization and summing properties of the Sobolev embedding operator. We propose two different approaches. One shows that the Sobolev embedding operator  $S:W^{1,1}(\mathbb{T}^2)\hookrightarrow L_2(\mathbb{T}^2)$ factorises through the identical embedding $\ell_\Phi\hookrightarrow\ell_2$ for some Young function with Matuszewska-Orlicz index 1. Proof of this fact is based on two results of independent interest. First, a necessary and sufficient conditions on a Young function $\Phi$ and weight $\Psi$ for boundedness of the embedding of the Sobolev space $W^{1,1}(\mathbb{T}^2)$ into Besov-Orlicz space $B^\Psi_{\Phi,1}(\mathbb{T}^2)$. Second, a generalization of the Marcinkiewicz sampling theorem to the context of Orlicz spaces. Another approach is based on the extrapolation of $(p,1)$-summing norm.
\end{abstract}

\let\thefootnote\relax\footnotetext{This research was partially supported by the National Science Centre, Poland, and Austrian Science Foundation FWF joint CEUS programme. National Science Centre project no. 2020/02/Y/ST1/00072 and FWF project no. I5231.}
\section{Introduction}
This paper is devoted to the study of absolutely summing properties of the Sobolev embedding operator. We generalise previously known results and, as a consequence, obtain new relaying on the theory of $(\Phi,1)$-summing operator developed in papers \cite{MaMa} and \cite{DeMaMi}. 
The starting point of our considerations is the following result concerning the operator ideal properties of the Sobolev embedding operator which was proved in \cite{Woj}:
\begin{theorem}\label{posit}
Let $d=2, 3, ...$; $k=1, 2, ..., d-1$; $1\leq p< 2$ and $p<d/k$. Then the Sobolev embedding $\mathcal{S}_{d,k,p}:W^{k,p}(\mathbb{T}^d)\hookrightarrow L_s(\mathbb{T}^d)$, where $1/s=1/p-k/d$, is $(v,1)$-absolutely summing for $v>\max (2d/(2k+d),p)$.
\end{theorem}

In particular, setting in the above theorem $d=2, k=1, p=1$ and $s=2$ we get that the operator
$$\mathcal{S}_{2,1,1}:W^{1,1}(\mathbb{T}^2)\hookrightarrow L_2(\mathbb{T}^2)$$
is $(v,1)$-summing for every $v>1$.

The proof of the above Theorem used the absolute $(v,1)$-summability property of the embedding $\ell_p\hookrightarrow\ell_q$ for $1\leq p\leq q\leq 2$ and $1/v=1/p-1/q+1/2$, see \cite{Ben}, \cite{Car}.

Recall that a continuous linear operator $T:X \rightarrow Y$ between Banach sequence spaces is said to be $(\Phi,1)$-\textit{summing} if there exists a constant $C>0$ such that for all $x_1, ..., x_n\in X$,
$$\left\Vert(\,||Tx_i||_Y\,)_{i=1}^n\right\Vert_{\ell_\Phi}\leq C\sup_{x^*\in B_{X^*}}\sum_{i=1}^n |x^*(x_i)|,$$
where $B_{X^*}$ denotes a unit ball in the dual space of $X$. Then we will write $T\in\Pi_{\Phi,1}$. By $\pi_{\Phi,1}(T)$ we denote the smallest constant $C$ for which above inequality is satisfied. 

In this paper we prove two results. The first concerns the factorization of the Sobolev embedding through $Id:\ell_\Phi\hookrightarrow\ell_2$ embedding. We find the Orlicz space with Orlicz-Matuszewska index equal to one for which the factorization holds. This in turn by the Maligranda-Masty{\l}o generalization of the Bennet-Carl Theorem yields $(\Phi,1)$-summability of the Sobolev embedding. On the other hand we show that the much stronger summability result could be obtained while we don't relate the summability to factorization. In fact we provide very general extrapolation methods which allows to derive $(\Phi,1)$-summability from the asymptotic behaviour of $(p,1)$-summability norms.
While the factorization result is based on a delicate study of embedding of Sobolev space into Besov-Orlicz spaces and the Orlicz space version of Marcinkiewicz sampling theorem, the extrapolation result is of abstract nature and could be applied to a wide class of operators.  

In Section 2 we introduce basic notations and definitions used in the paper. Also we quote there the results needed in the subsequent sections.  

In Sections 3-7 we determine the Young function $\Phi$ with Matuszewska-Orlicz index equal 1 (then $\ell_\Phi$ is smaller than any $\ell_p,\,p>1$) such that the Sobolev embedding factorises through $Id:\ell_\Phi\hookrightarrow\ell_2$.

In Section 8, based on the asymptotic estimate of $(p,1)$-summing norm $\Pi_{p,1}$ of operator $\mathcal{S}_{d,k,p}$ from Theorem \ref{posit}, we will prove that the Sobolev embedding $\mathcal{S}_{d,k,p}:W^{k,p}(\mathbb{T}^d)\hookrightarrow L_s(\mathbb{T}^d)$ is $(\Phi,1)$-summing for $\Phi(x)=\frac{x^{p_0}}{|\ln(x)|^{\gamma}}$, for suitable chosen  $p_0$ and $\gamma$. The proof however gives no information about the factorisation of the Sobolev embedding through the inclusion map $\ell_\Phi\hookrightarrow\ell_p$ which was the core of the proof of Theorem \ref{glowne_twierdzenie}.

First of the results is the following theorem.
\begin{theorem}
\label{glowne_twierdzenie}
The embedding of the space $W^{1,1}(\mathbb{T}^2)$ into the space $L_2(\mathbb{T}^2)$ factorizes through $\ell_\Phi\hookrightarrow\ell_2$ if $\Phi$ is a Young function such that
\begin{enumerate} 
\item there is a constant $C$ such that for every $s\geq1$ we have
$$\frac{s}{\Phi^{-1}(s^2)}\int_1^s\frac{\Phi^{-1}(t^2)}{t^2}\, dt+\int_s^\infty\frac{\Phi^{-1}(t^2)s}{t^2\Phi^{-1}(ts)}\,dt< C,$$
\item the function $x\mapsto \Phi(\sqrt{x})$, $x\in\mathbb{R}_+$, is concave,
\item there is a constant $C$ such that $\Phi(a)\cdot\Phi(b)\leq\Phi(Cab)$ for all real numbers $0<a<1\leq ab<b$,
\item there is a constant $C$ such that $1\leq C\Phi^{-1}(x)\Phi^{-1}(1/x)$ for all positive numbers $x$.
\end{enumerate}
\end{theorem}

Let us recall the following result by Maligranda and Masty\l{}o \cite[Theorem 2.3 (i), page 271]{MaMa}.
\begin{theorem}
\label{Mastylo}
Let $\Phi$ be a Young function. If $x\mapsto\Phi(\sqrt{x})$ is concave on $\mathbb{R}_+$, then the inclusion map $\ell_\Phi\hookrightarrow\ell_2$ is $(\Phi,1)$-summing.
\end{theorem}

The immediate corollary of Theorem \ref{glowne_twierdzenie} and Theorem \ref{Mastylo} is the following.
\begin{corollary}
\label{Corol}
If $\Phi$ satisfies the conditions of Theorem \ref{glowne_twierdzenie}, then the Sobolev embedding $W^{1,1}(\mathbb{T}^2)\hookrightarrow L_2(\mathbb{T}^2)$ is $(\Phi,1)$-summing.
\end{corollary}


The proof of Theorem \ref{glowne_twierdzenie} is divided into several steps. The diagram below presents the idea.

\begin{equation}\label{diagram}
\begin{tikzcd}
	W^{1,1}(\mathbb{T}^2) \arrow[r, hook,"A"] & B_{\Phi,1}^\Psi(\mathbb{T}^2) \arrow[r, hook,"B"] & \widetilde{B_{\Phi,1}^\Psi}(\mathbb{T}^2) \arrow[r, hook,"C"] \arrow[d, hook,"D"] & \widetilde{B_{2,2}^0}(\mathbb{T}^2) \arrow[r, hook,"G"]& L^2(\mathbb{T}^2) \\
	& & \ell_\Phi \arrow[r, hook, "E"] & \ell_2 \arrow[u, hook,"F"] &
\end{tikzcd}
\end{equation}

In Section 3, \textit{Embedding of the Sobolev space into the Besov-Orlicz space}, we deal with the inclusion map $A$ in the above diagram. The embedding $B$ is described in Section 4, \textit{Comparison of two norms defining the Besov-Orlicz space}. For the embedding $C$ we use the factorization. We, in fact, prove that the inclusion maps $D$, $E$ and $F$ hold. All this is described in Section 5, \textit{Marcinkiewicz type sampling theorem}, and Section 6, \textit{Factoring through embedding $\ell_\Phi\hookrightarrow\ell_2$}. Operator $E$ is $(\Phi,1)$-summing in our case by Theorem \ref{Mastylo}. The embedding $G$ is a trivial one. 

The two buildings blocks of Theorem \ref{glowne_twierdzenie}, are of particular interests. First one describes whether the Sobolev space $W^{1,1}(\mathbb{T}^d)$ can be embedded into the Besov-Orlicz space $B_{\Phi,1}^\Psi(\mathbb{T}^d)$. In \cite{PaWoj2} a stronger version of this theorem is presented. Similar integral condition is given which implies the embedding of the space of functions of bounded variation into Besov-Orlicz space in the case when the underlying space is $\mathbb{R}^d$ or a compact subset of $\mathbb{R}^d$. In the second case the integral condition is equivalent to embedding. For the readers convenience we give here the statement of the theorem appropriate for our purpose and its proof is given in Section 3. 

\begin{theorem}
\label{Kolyada}
Let $B_{\Phi,1}^\Psi(\mathbb{T}^d)$ be a Besov-Orlicz space for some increasing continuous function $\Psi$ and a Young function $\Phi$. Let us assume that $L_{d/(d-1)}(\mathbb{T}^d)\hookrightarrow L_\Phi(\mathbb{T}^d)$. Then the Sobolev space $W^{1,1}(\mathbb{T}^d)$ can be continuously embedded into the space $B_{\Phi,1}^\Psi(\mathbb{T}^d)$ if and only if there exists a constant $C>0$ such that for every $s\geq 1$ we have
\begin{eqnarray}\label{Kolyada_nier}
\frac{s^{d-1}}{\Phi^{-1}(s^d)}\int_1^{s}\frac{\Psi(t)}{t}dt+ \int_{s}^\infty\frac{\Psi(t)s^{d-1}}{\Phi^{-1}(ts^{d-1})t} dt<C.
\end{eqnarray}
\end{theorem}

This theorem stands behind the arrow $A$ in diagram \eqref{diagram}. The second ingredient concerns the operator $D$ in diagram \eqref{diagram}. It generalizes the Marcinkiewicz sampling theorem (see \cite[Volume II, Theorem 7.5, page 28]{Zyg}) and reads as follows.

\begin{theorem}
\label{podmultiplikatywnosc}
Let $\Phi$ be a Young function and $C\geq 1$. Assume that $\Phi$ fulfils the conditions:
\begin{enumerate}
    \item for all real numbers $a$ and $b$ such that $0<a<1\leq ab<b$ we have
    $$\Phi(a)\Phi(b)\leq\Phi(Cab),$$
    \item for every positive $x$
    $$1\leq C\Phi^{-1}(x)\Phi^{-1}(1/x).$$
\end{enumerate}
Then for every  trigonometric polynomial $g$ such that $\mbox{supp}\,\hat{g}\subseteq G_n$ for $n=0, 1, 2, ...$ we have
\begin{eqnarray}\label{podmultiplikatywnosc_nier}
||(g)||_{\ell_\Phi}\leq 24C^2\,\Phi^{-1}\left(\omega_n\right)||g||_{L_\Phi},
\end{eqnarray}
where $\omega_n$ is the number of elements of the set $G_n$ (for the definition of $G_n$ see \eqref{G} below). The definition of the norm $||(g)||_{\ell_\Phi}$ is given at the end of Section 1, see \eqref{small_norm}.
\end{theorem}

In the paper \cite{PaWoj} the one dimensional theorem is established. Interested reader should consult above paper for better explanation and finer properties of sampling phenomenon for Orlicz spaces.

In Section 7 we provide an example of Orlicz function satisfying assumptions of Theorem \ref{glowne_twierdzenie}.

In Section 8, which could be read independently on the previous ones, extrapolation method is used to prove the following summability result. 

\begin{theorem}\label{invariantnafaktor}
 Let $d\in \mathbb{N}\setminus \{0,1\}$, $k\in\{1,2,\dots, d-1\}$, $1\leq p < 2$ and $p<d/k$. Then the Sobolev embedding $S_{d,k,p}: W^{k,p}(\mathbb{T}^d)\rightarrow L_s(\mathbb{T}^d)$, where $1/s=1/p-k/d$, is $(\Phi,1)$-summing operator for $\Phi(x)=\frac{x^{p_0}}{|\ln(x)|^{\gamma}}$ and $0\leq x \leq \frac{1}{2}$, where $p_0=\max\{\frac{2d}{2k+d},p\}$, $\gamma>p_0(\frac{2}{p}-1)$. 
\end{theorem}

Note that the function $\Phi$ from Theorem \ref{invariantnafaktor} defines larger Orlicz Space then the example obtained in Section 7.

\section{Definitions and notation}
By \textit{Young function} we mean a continuous, strictly increasing, convex function $\Phi:[0,\infty)\rightarrow \mathbb{R}_+$ such that $\Phi(0)=0$, $\lim_{t\rightarrow\infty}t/\Phi(t)=0$ and $\lim_{t\rightarrow 0}\Phi(t)/t=0$. For a complexiv exposition of the theory of Orlicz spaces see \cite{KrRu} or \cite{RaRe}.

The Sobolev space $W^{1,1}(\mathbb{T}^2)$ is the space of functions  from $L_1(\mathbb{T}^2)$ which have a distributional gradient belonging to $L_1(\mathbb{T}^2)$: 
$$W^{1,1}(\mathbb{T}^2)=\left\{f\in L_1(\mathbb{T}^2): \nabla f\in L_1(\mathbb{T}^2)\right\}$$
equipped with the norm
$$||f||_{W^{1,1}}=||f||_{L_1}+||\nabla f||_{L_1}.$$

We will prove that the embedding
$$W^{1,1}(\mathbb{T}^2)\hookrightarrow L_2(\mathbb{T}^2)$$
factorizes through $\ell_\Phi\hookrightarrow\ell_2$ for some Young functions $\Phi$ which dominate near zero every $x^p$ function, $p>1$, meaning that 
$$\lim_{t\rightarrow 0}\frac{\Phi(t)}{t}=\lim_{t\rightarrow 0}\frac{t^p}{\Phi(t)}=0,$$
for every $p>1$.

Now we give basic definitions. For the Young function $\Phi$ and for some increasing continuous function $\Psi$ we define the generalized Besov-Orlicz space:
$$B_{\Phi,1}^\Psi=\left\{f\in L_\Phi(\mathbb{T}^2):\sum_{n=0}^\infty \Psi(2^n)\omega_\Phi(f,2^{-n})<\infty\right\},$$
where $L_\Phi(\mathbb{T}^2)$ is an ordinary Orlicz space of integrable functions on two dimensional torus with the Luxemburg norm
$$||f||_{L_\Phi}=\inf\left\{\lambda>0:\frac{1}{(2\pi)^2}\int_{\mathbb{T}^2}\Phi\Big(\frac{|f(x)|}{\lambda}\Big)\,dx\leq 1\right\}$$
and
$$\omega_\Phi(f,t)=\sup_{|h|\leq t}\left\Vert f(\cdot+h)-f(\cdot)\right\Vert_{L_\Phi}$$
being the integral modulus of continuity.
We define the norm on the space $B_{\Phi,1}^\Psi$ by
$$||f||_{B_{\Phi,1}^\Psi}=||f||_{L_\Phi}+\sum_{n=0}^\infty \Psi(2^n)\omega_\Phi(f,1/2^n).$$

More information on Besov-Orlicz space can be found in a classical nowadays paper \cite{PiSi}.
	
In order to prove one of the main results, we need two different norms on the Besov-Orlicz space. One is needed to show the embedding of the Sobolev space into the Besov-Orlicz space, and the other to prove that the embedding of the Besov-Orlicz space into the Hilbert space, for appropriate Young function $\Phi$, factorizes through $\ell_\Phi\hookrightarrow\ell_2$. We will show the relevant relationships between these norms. Before we define the second Besov-Orlicz norm we need a short preparation.
	
The following construction of a partition of unity is complicated a bit. The reason is that in the underlying Orlicz space there need not exist a bounded projection on the space of trigonometric polynomials (in other words, the Hilbert transform need not be a bounded operator on this Orlicz space). This is the difference between this construction and the construction of similar Besov norm in paper \cite{Woj}, where authors can deal with the characteristic functions of frames.

The one dimensional Fej\'er kernel is a function $F_n:\mathbb{T}\rightarrow\mathbb{R}_+$, $n=0, 1 ,2, 3, ...$, defined as
$$F_n(x)=\sum_{k=-n}^n\Big(1-\frac{|k|}{n+1}\Big)e^{ikx},$$
while the two-dimensional Fej\'er kernel is a function $F_{m,n}:\mathbb{T}^2\rightarrow\mathbb{R}_+$, $m, n=1, 2, 3, ...$, such that
$$F_{m,n}(x,y)=F_m(x)\cdot F_n(y).$$

Now we define a partition of unity for the space $\mathbb{Z}\times\mathbb{Z}$. Let $f_k:\mathbb{T}^2\rightarrow\mathbb{C}$, $k\in\mathbb{N}_+$, be a function defined as
$$f_k(x,y)=\big[F_{2^k-1}(x)\big(e^{-i2^kx}+1+e^{i2^kx}\big)\big]\cdot\big[F_{2^k-1}(y)\big(e^{-i2^ky}+1+e^{i2^ky}\big)\big]$$
and we define
$$f_{-1}(x,y)=f_0(x,y)=F_{0,0}(x,y).$$
Notice that
$$\mbox{supp}\widehat{f_k}=\{(m,n\in\mathbb{Z}^2: \widehat{f_k}(m,n)\neq 0\}\subseteq (-2^{k+1},2^{k+1})\times(-2^{k+1},2^{k+1}).$$

Define
$$g_0(x,y)=f_{-1}(x,y),\,\, g_1(x,y)=f_0(x,y),$$
and
$$g_{k+1}(x,y)=f_k(x,y)-f_{k-2}(x,y),$$
for $k=1, 2, 3, ...$ and let
\begin{eqnarray}\label{G}
G_k=\{(m,n)\in\mathbb{Z}^2: \widehat{g_k}(m,n)\neq 0\}.
\end{eqnarray}
Then a partition of unity for $\mathbb{Z}\times\mathbb{Z}$ is a sequence of functions $\{\hat{g}_k\}_{k=0}^\infty$.
By $\omega_n$ we denote the number of elements of $G_n$. In the sequel we will need the $L_1$ norm estimate of the functions $g_k$, $k=0, 1, 2, ...\,$. Notice that
$$||g_k||_{L_1}\leq18.$$

We can now define the second norm on the Besov-Orlicz space.
For the Young function $\Phi$, increasing continuous function $\Psi$ and function $f\in L_\Phi$ we put
$$||f||_{\widetilde{B_{\Phi,1}^\Psi}}=||f||_{L_\Phi}+\sum_{n=0}^\infty \Psi(2^n)||g_n*f||_{L_\Phi}.$$
For $\theta\in\mathbb{R}$ and $p,q \geq 1$ we put 
\[
||f||_{\widetilde{B_{p,q}^s}}=\left(||f||_{L_p}+\sum_{n=0}^\infty 2^{qns}||g_n*f||^q_{L_p}\right).
\]
For the purposes of this article we will use the following convention. For the Young function $\Phi$ and a given function $f$ on $\mathbb{T}^2$ we define the norm 
\begin{equation}
\label{small_norm}
||(f)||_{\ell_\varphi}=\inf\left\{\lambda>0: \sum_{(k,l)\in G_n}\varphi\left(\frac{|f(x_k,y_l)|}{\lambda}\right)\leq 1\right\},
\end{equation}
where $G_n=\mathbb{Z}^2\cap\left[ \left(-2^n, 2^n\right)\times\left(-2^n, 2^n\right) \setminus \left[-2^{n-3}, 2^{n-3}\right]\times\left[-2^{n-3}, 2^{n-3}\right]\right]$ is a "frame" and 
$$x_k=2\pi\cdot\frac{k+2^n-1}{2^{n+1}-1} \mbox{ and } y_l=2\pi\cdot\frac{l+2^n-1}{2^{n+1}-1}.$$
The symbol 
$a\lesssim b$
will mean that there is a constant $C>0$ such that 
$a\leq C b.$

\section{Embedding of Sobolev space into Besov-Orlicz space}

In this section we will prove the generalization of the embedding of certain Sobolev spaces into Besov spaces. In brief, we will prove that the Sobolev space $W^{1,1}(\mathbb{T}^d)$ can be embedded continuously into the space $B_{\Phi,1}^\Psi(\mathbb{T}^d)$ for an appropriate functions $\Phi$ and $\Psi$. We will denote by $\mu_d$ a $d$-dimensional Lebesgue measure.

The following lemma is a simple observation, which we will leave without proof.

\begin{lemma}
\label{lemacik1}
Let $\Phi$ be a Young function and let $\Psi$ be some increasing continuous function. Then for every $f\in B_{\Phi,1}^\Psi(\mathbb{T}^d)$ we have
$$\frac{1}{2}\int_1^\infty\frac{\Psi(t)}{t}\omega_\Phi(f,1/(2t))\,dt\leq\sum_{n=0}^\infty\Psi(2^n)\omega_\Phi(f,1/2^n)\leq 2\int_1^\infty\frac{\Psi(t)}{t}\omega_\Phi(f,2/t)\,dt.$$
\end{lemma}
This lemma gives an estimate of the norm:
$$||f||_{B_{\Phi,1}^{\Psi}}=||f||_{L_\Phi}+\sum_{n=0}^\infty \Psi(2^n)\omega_\Phi(f,1/2^n)\leq
||f||_{L_\Phi}+2\int_1^\infty\frac{\Psi(t)}{t}\omega_\Phi(f,2/t)\,dt.$$

In the proof we will need a molecular decomposition of the Sobolev space $W^{1,1}(\mathbb{T}^d)$. The theorem below comes from article \cite[Theorem $2.1_\mathbb{T}$, page 82]{PeWoj}.
\begin{theorem}
\label{molecular}
Let $d\in\mathbb{N}_+$. There exist positive constants $a$ and $b$ such that for every $\epsilon>0$, given a function $f\in W^{1,1}(\mathbb{T}^d)$ there exists a sequence $(f_m)_{m=1}^\infty\subseteq W^{1,1}(\mathbb{T}^d)$ such that 
\begin{eqnarray}\label{molecular1}
\sum_{m=1}^\infty f_m(x)=f(x),\quad \sum_{m=1}^\infty\nabla f_m(x)=\nabla f(x)\quad   \mbox{$\mu_d$-a.e.}
\end{eqnarray}
\begin{eqnarray}\label{molecular2}
\sum_{m=1}^\infty||f_m||_{L_1}\leq(1+\epsilon)a||f||_{L_1},\quad  \sum_{m=1}^\infty||\nabla f_m||_{L_1}\leq(1+\epsilon)a||\nabla f||_{L_1},
\end{eqnarray}
\begin{eqnarray}\label{molecular3}
||f_m||_\infty^{1/d}||f_m||_{L_1}^{1-1/d}\leq(1+\epsilon)b\big(||\nabla f_m||_{L_1}+||f_m||_{L_1}\big)\quad \mbox{ for } m=1, 2, ... .
\end{eqnarray}
\end{theorem}

The proof of Theorem \ref{Kolyada} is a generalization of the ideas contained in \cite{PeWoj}, where similar reasoning was carried out for the $L_p$ norms. In fact in \cite{PeWoj} the following statement was proved.

\begin{theorem*}
Let $d=1, 2, ...$. Then 
$$W^{1,1}(\mathbb{R}^d)\hookrightarrow B_{p,1}^{\theta(p,d)}(\mathbb{R}^d),$$
where $\theta(p,d)=d(1/p+1/d-1)$ and $1<p<d/(d-1)$.
\end{theorem*}

The last ingredient of the proof of Theorem \ref{Kolyada} is a geometrical lemma. We will use it to prove necessity of our condition. 
\begin{lemma}\label{geometrical}
Let $\mathcal{B}_d(0,r)$ be a closed $d$-dimensional ball centred at zero and of radius $r>0$. Moreover let $V_d$ denote the volume of $\mathcal{B}_d(0,1)$,
$$V_d=\mu_d\big(\mathcal{B}_d(0,1)\big)$$
and let $\alpha$ be a real number such that $0\leq\alpha<r$. Then
$$\mu_d\Big(\big(\mathcal{B}_d(0,r)\cup\mathcal{B}_d(x,r)\big)\setminus\big(\mathcal{B}_d(0,r)\cap\mathcal{B}_d(x,r)\big)\Big)\geq V_dr^{d-1}\alpha,$$
where $x$ is a point of $\mathbb{R}^d$ such that $|x|=2\alpha$.
\end{lemma}
We leave the proof of Lemma 2 for the reader. Now we are ready to prove Theorem \ref{Kolyada}.

\begin{proof}[Proof of Theorem \ref{Kolyada}]
At first we will prove sufficiency of condition \eqref{Kolyada_nier}. We estimate the value 
$$\int_1^\infty\frac{\Psi(t)}{t}\omega_\Phi(f,2/t)\,dt,$$
for $f\in W^{1,1}(\mathbb{T}^d)$.
In order to do that, let us use the molecular decomposition. Let
$$f=\sum_{m=1}^\infty f_m,$$
as in point \eqref{molecular1} of the Theorem \ref{molecular}. We have
\begin{eqnarray*}
\int_1^\infty\frac{\Psi(t)}{t}\omega_\Phi(f,2/t)\,dt
&=&
\int_1^\infty\frac{\Psi(t)}{t}\omega_\Phi\Big(\sum_{m=1}^\infty f_m,2/t\Big)\,dt \\
&\leq&
\sum_{m=1}^\infty\int_1^\infty\frac{\Psi(t)}{t}\omega_\Phi(f_m,2/t)\,dt.
\end{eqnarray*}
Now we will estimate each summand separately.

At the beginning we make a simple observation. The set
$$\Big\{\lambda>0:\frac{1}{(2\pi)^d}\int_{\mathbb{T}^d}\Phi\Big(\frac{|f_m(x+h)-f_m(x)|}{\lambda}\Big)\,dx\leq 1\Big\}$$ 
contains the set
$$\Big\{\lambda>0:\frac{1}{(2\pi)^d}\int_{\mathbb{T}^d}\frac{|f_m(x+h)-f_m(x)|}{2||f_m||_{L_\infty}}\Phi\Big(\frac{2||f_m||_{L_\infty}}{\lambda}\Big)\,dx\leq 1\Big\},$$
because the inequality $\Phi(\alpha a)\leq\alpha\Phi(a)$ is true for all $a>0$, $\alpha\in[0,1]$ and a convex function $\Phi$ such that $\Phi(0)=0$.
Computing the infima of the above sets we get
\begin{eqnarray}\label{infima}
||f_m(\cdot+h)-f_m(\cdot)||_{L_\Phi}\leq\frac{2||f_m||_{L_\infty}}{\Phi^{-1}\Big(\frac{2||f_m||_{L_\infty}}{||f_m(\cdot+h)-f_m(\cdot)||_{L_1}}\Big)}.
\end{eqnarray}
Now, we take supremum over $|h|<2/t$ and get
$$\omega_\Phi(f_m,2/t)\leq\frac{2||f_m||_{L_\infty}}{\Phi^{-1}\Big(\frac{2||f_m||_{L_\infty}}{\omega_{L_1}(f_m,2/t)}\Big)}.$$
By \cite[Lemma 3.5, page 94]{PeWoj} we have
$$\omega_{L_1}(f_m,2/t)\leq\frac{2}{t}||\nabla f_m||_{L_1}.$$
Therefore
\begin{eqnarray*}
\omega_\Phi(f_m,2/t)
&\leq&
\frac{2||f_m||_{L_\infty}}{\Phi^{-1}\Big(\frac{t||f_m||_{L_\infty}}{||\nabla f_m||_{L_1}}\Big)} \\
&\leq&
\frac{2||f_m||_{L_\infty}}{\Phi^{-1}\Big(\frac{t||f_m||_{L_\infty}}{||\nabla f_m||_{L_1}+||f_m||_{L_1}}\Big)}.
\end{eqnarray*}

Now let us provide the estimate of the integrand which will be used for "big" values $t$. We have
\begin{equation}\label{equati1}
\begin{aligned}
\frac{\Psi(t)}{t}\omega_\Phi(f_m,&2/t)
\leq
\frac{\Psi(t)}{t}\frac{2||f_m||_{L_\infty}}{\Phi^{-1}\Big(\frac{t||f_m||_{L_\infty}}{||\nabla f_m||_{L_1}+||f_m||_{L_1}}\Big)} \\
&=
2\frac{\Psi(t)}{t}\frac{||f_m||_{L_\infty}}{\big(||\nabla f_m||_{L_1}+||f_m||_{L_1}\big)\Phi^{-1}\Big(\frac{t||f_m||_{L_\infty}}{||\nabla f_m||_{L_1}+||f_m||_{L_1}}\Big)} \big(||\nabla f_m||_{L_1}+||f_m||_{L_1}\big).
\end{aligned}
\end{equation}

For "small" values $t$ we proceed in a different way. 
Let us assume for the moment that the dimension $d$ is at least $2$. By \eqref{infima} and the inequality $||f_m(\cdot+h)-f_m(\cdot)||_{L_1}\leq 2||f_m||_{L_1}$ we get
$$||f_m(\cdot+h)-f_m(\cdot)||_{L_\Phi}\leq\frac{2||f_m||_{L_\infty}}{\Phi^{-1}\Big(\frac{||f_m||_{L_\infty}}{||f_m||_{L_1}}\Big)}.$$

By \eqref{molecular3} we get
$$\frac{||f_m||_{L_\infty}}{||f_m||_{L_1}}=\frac{||f_m||_{L_\infty}^{d/(d-1)}}{||f_m||_{L_1}||f_m||_{L_\infty}^{1/(d-1)}}\geq\frac{1}{C}\Big(\frac{||f_m||_{L_\infty}}{||\nabla f_m||_{L_1}+||f_m||_{L_1}}\Big)^{d/(d-1)},$$
for absolute constant $C\geq 1$. Hence
\begin{eqnarray*}
||f_m(\cdot+h)-f_m(\cdot)||_{L_\Phi}
&\leq&
\frac{2||f_m||_{L_\infty}}{\Phi^{-1}\Big(\frac{1}{C}\big(\frac{||f_m||_{L_\infty}}{||\nabla f_m||_{L_1}+||f_m||_{L_1}}\big)^{d/(d-1)}\Big)} \\
&\leq&
C\frac{2||f_m||_{L_\infty}}{\Phi^{-1}\Big(\big(\frac{||f_m||_{L_\infty}}{||\nabla f_m||_{L_1}+||f_m||_{L_1}}\big)^{d/(d-1)}\Big)}.
\end{eqnarray*}
So we can write
\begin{equation}\label{equati2}
\begin{aligned}
\frac{\Psi(t)}{t}\omega_\Phi(f_m,2/t)
&=
\frac{\Psi(t)}{t}\sup_{|h|<2/t}||f_m(\cdot+h)-f_m(\cdot)||_{L_\Phi} \\
&\leq
\frac{\Big[C\frac{\Psi(t)}{t}\big(||\nabla f_m||_{L_1}+||f_m||_{L_1}\big)\Big]2||f_m||_{L_\infty}}{\big(||\nabla f_m||_{L_1}+||f_m||_{L_1}\big)\Phi^{-1}\Big(\big(\frac{||f_m||_{L_\infty}}{||\nabla f_m||_{L_1}+||f_m||_{L_1}}\big)^{d/(d-1)}\Big)}.
\end{aligned}
\end{equation}

Now we define
$$s_m=\Big(\frac{||f_m||_{L_\infty}}{||\nabla f_m||_{L_1}+||f_m||_{L_1}}\Big)^{1/(d-1)}.$$
Let $\mathbb{N}=A\cup B$, where $A=\{m\in\mathbb{N}: s_m\geq 1\}, B=\mathbb{N}\setminus A$.
We get
\begin{eqnarray*}
\int_1^\infty\frac{\Psi(t)}{t}\omega_\Phi(f,2/t)\, dt
&\leq&
\sum_{m=1}^\infty\int_1^\infty\frac{\Psi(t)}{t}\omega_\Phi(f_m,2/t)\, dt \\
&=&
\sum_{m\in A}\int_1^\infty\frac{\Psi(t)}{t}\omega_\Phi(f_m,2/t)\, dt +\sum_{m\in B}\int_1^\infty\frac{\Psi(t)}{t}\omega_\Phi(f_m,2/t)\, dt \\
&=&
I+II.
\end{eqnarray*}

For sum $I$ we use \eqref{equati1} and \eqref{equati2} to get
\begin{eqnarray*}
I
&=&
\sum_{m\in A}\int_1^\infty\frac{\Psi(t)}{t}\omega_\Phi(f_m,2/t)\, dt \\
&\lesssim&
\sum_{m\in A}\Big(\int_1^{s_m}\frac{\Psi(t)}{t}\frac{s_m^{d-1}}{\Phi^{-1}(s_m^d)}\,dt+\int_{s_m}^\infty\frac{\Psi(t)}{t}\frac{s_m^{d-1}}{\Phi^{-1}(s_m^{d-1}t)}\,dt \Big)\big(||\nabla f_m||_{L_1}+||f_m||_{L_1}\big) \\
&\lesssim& 
\sum_{m\in A}\big(||\nabla f_m||_{L_1}+||f_m||_{L_1}\big),
\end{eqnarray*}
based on the assumption of Theorem \ref{Kolyada}.

If the value $\frac{||f_m||_{L_\infty}}{||\nabla f_m||_{L_1}+||f_m||_{L_1}}$ is smaller than one, then we can use the inequality $\Phi^{-1}(\alpha a)\geq\alpha\Phi^{-1}(a)$, which is true for $a\geq 0$, $\alpha\in [0,1]$ and for concave function $\Phi^{-1}$ such that $\Phi^{-1}(0)=0$. We then obtain
\begin{eqnarray*}
\int_1^\infty\frac{\Psi(t)}{t}\omega_\Phi(f_m,2/t)\, dt
&\leq&
\int_1^\infty\frac{\Psi(t)}{t}\frac{2||f_m||_{L_\infty}}{\Phi^{-1}\Big(\frac{t||f_m||_{L_\infty}}{||\nabla f_m||_{L_1}+||f_m||_{L_1}}\Big)}\,dt \\
&\leq&
2\int_1^\infty\frac{\Psi(t)}{t}\frac{1}{\Phi^{-1}(t)}\,dt \big(||\nabla f_m||_{L_1}+||f_m||_{L_1}\big).
\end{eqnarray*}

This gives us an estimate of $II$.
\begin{eqnarray*}
II
&=&
\sum_{m\in B}\int_1^\infty\frac{\Psi(t)}{t}\omega_\Phi(f_m,2/t)\, dt \\
&\lesssim&
2\sum_{m\in B}\int_1^\infty\frac{\Psi(t)}{t}\frac{1}{\Phi^{-1}(t)}\,dt \big(||\nabla f_m||_{L_1}+||f_m||_{L_1}\big) \\
&\lesssim&
\sum_{m\in B}\big(||\nabla f_m||_{L_1}+||f_m||_{L_1}\big),
\end{eqnarray*}
by \eqref{Kolyada_nier} for $s=1$. 

We can sum up the above estimates and write
\begin{eqnarray*}
\int_1^\infty\frac{\Psi(t)}{t}\omega_\Phi(f,2/t)\, dt
&=&
I+ II \\
&\lesssim&
\sum_{m\in\mathbb{N}}\big(||\nabla f_m||_{L_1}+||f_m||_{L_1}\big) \\
&\lesssim&
||\nabla f||_{L_1}+||f||_{L_1},
\end{eqnarray*}
by inequality \eqref{molecular2} of Theorem \ref{molecular}.

In dimension $d=1$ we know from Theorem \ref{molecular}, point \eqref{molecular3} that there exists constant $C>1$ such that for every $m=1, 2, ...$ we have
$$||f_m||_\infty\leq C\big(||\nabla f_m||_{L_1}+||f_m||_{L_1}\big).$$
Thus, based on the above calculations we can write
\begin{eqnarray*}
\int_1^\infty\frac{\Psi(t)}{t}\omega_\Phi(f_m,2/t)\, dt
&\leq&
\int_1^\infty\frac{\Psi(t)}{t}\frac{2||f_m||_{L_\infty}}{\Phi^{-1}\Big(\frac{t||f_m||_{L_\infty}}{||\nabla f_m||_{L_1}+||f_m||_{L_1}}\Big)}\,dt \\
&\leq&
\int_1^\infty\frac{\Psi(t)}{t}\frac{2||f_m||_{L_\infty}}{\Phi^{-1}\Big(\frac{t||f_m||_{L_\infty}}{C\big(||\nabla f_m||_{L_1}+||f_m||_{L_1}\big)}\Big)}\,dt \\
&\leq&
\int_1^\infty\frac{\Psi(t)}{t}\frac{2||f_m||_{L_\infty}}{\Phi^{-1}(t)\frac{||f_m||_{L_\infty}}{C\big(||\nabla f_m||_{L_1}+||f_m||_{L_1}\big)}}\,dt \\
&\leq&
2C\int_1^\infty\frac{\Psi(t)}{t}\frac{1}{\Phi^{-1}(t)}\,dt \big(||\nabla f_m||_{L_1}+||f_m||_{L_1}\big).
\end{eqnarray*}
Notice also that by assumption of the theorem we have
$$\lVert f\rVert_{L_\Phi}\lesssim \lVert f\rVert_{L_{d/(d-1)}}\lesssim\lVert f\rVert_{W^{1,1}}.$$

\bigskip
 The necessity of condition \eqref{Kolyada_nier} will be proved by contradiction. Let us assume that there exists a constant $D>0$ such that 
\begin{eqnarray}\label{eq7}
||f||_{B_{\Phi,1}^\Psi}\leq D||f||_{W^{1,1}}
\end{eqnarray}
for every $f\in W^{1,1}(\mathbb{T}^d)$ and that for every $C>0$ there exist $s\geq 1$ such that 
$$\frac{s^{d-1}}{\Phi^{-1}(s^d)}\int_1^{s}\frac{\Psi(t)}{t}dt+ \int_{s}^\infty\frac{\Psi(t)s^{d-1}}{\Phi^{-1}(ts^{d-1})t} dt\geq C.$$
Assume also that $0<r<\frac{1}{8\pi}$ and $\varepsilon>0$ is a small number. We define a function $f\in W^{1,1}(\mathbb{T}^d)$ by putting $f(x)=1$ for $x\in\mathcal{B}(0,2\pi r)$ and $f(x)=0$ for $x\not\in\mathcal{B}(0,2\pi(r+\varepsilon))$ ($\varepsilon$ is a small positive number), and on the remaining part of the domain we define it to be continuous and linear on each segment of the form $[s,\frac{2\pi(r+\varepsilon)}{2\pi r}s]$, where $s\in\partial B(0,2\pi r)$.

Now we can compute appropriate norms. We have
\begin{eqnarray*}
||f||_{W^{1,1}}
&=&
||f||_{L_1}+||\nabla f||_{L_1} \\
&=&
\frac{1}{(2\pi)^d}\int_{\mathbb{T}^d}|f(x)|\,dx+\frac{1}{(2\pi)^d}\int_{\mathbb{T}^d}|\nabla f(x)|\,dx \\
&\leq&
V_d(r+\varepsilon)^d+\frac{1}{\varepsilon}V_d\left[(r+\varepsilon)^d-r^d\right] \\
&\leq&
2dV_dr^{d-1}
\end{eqnarray*}
for sufficiently small $\varepsilon$. We have by Lemma \ref{geometrical} ($\chi_A$ denote the characteristic function of the set $A$)
\begin{eqnarray*}
\int_1^\infty\frac{\Psi(t)}{t}\omega_\Phi(f,\frac{1}{2t})\,dt
&\geq&
\int_1^\infty\frac{\Psi(t)}{t}\omega_\Phi(\chi_{\mathcal{B}(0,2\pi r)},\frac{1}{2t})\,dt \\
&=&
\int_1^\frac{1}{8\pi r}\frac{\Psi(t)}{t}\omega_\Phi\big(\chi_{\mathcal{B}(0,2\pi r)},\frac{1}{2t}\big)\,dt+\int_\frac{1}{8\pi r}^\infty\frac{\Psi(t)}{t}\omega_\Phi\big(\chi_{\mathcal{B}(0,2\pi r)},\frac{1}{2t}\big)\,dt \\
&\geq&
\int_1^\frac{1}{8\pi r}\frac{\Psi(t)}{t}\frac{1}{\Phi^{-1}\Big(\frac{1}{2V_d r^d}\Big)}\,dt
+
\int_\frac{1}{8\pi r}^\infty\frac{\Psi(t)}{t}\frac{1}{\Phi^{-1}\Big(\frac{8\pi t}{V_d r^{d-1}}\Big)}\,dt \\
&=&
\int_1^\frac{1}{8\pi r}\frac{\Psi(t)}{t}\frac{1}{\Phi^{-1}\Big(\frac{(8\pi)^d}{2V_d (8\pi r)^d}\Big)}\,dt
+
\int_\frac{1}{8\pi r}^\infty\frac{\Psi(t)}{t}\frac{1}{\Phi^{-1}\Big(\frac{(8\pi)^d t}{V_d (8\pi r)^{d-1}}\Big)}\,dt \\
&\geq&
\int_1^\frac{1}{8\pi r}\frac{\Psi(t)}{t}\frac{1}{\frac{(8\pi)^d}{2V_d}\Phi^{-1}\Big(\frac{1}{(8\pi r)^d}\Big)}\,dt
+
\int_\frac{1}{8\pi r}^\infty\frac{\Psi(t)}{t}\frac{1}{\frac{(8\pi)^d}{V_d}\Phi^{-1}\Big(\frac{t}{(8\pi r)^{d-1}}\Big)}\,dt \\
&=&
\frac{2dV_dr^{d-1}}{16\pi d}\left\{\int_1^\frac{1}{8\pi r}\frac{\Psi(t)\frac{1}{(8\pi r)^{d-1}}}{t\Phi^{-1}\Big(\frac{1}{(8\pi r)^d}\Big)}\,dt
+
\int_\frac{1}{8\pi r}^\infty\frac{\Psi(t)\frac{1}{(8\pi r)^{d-1}}}{t\Phi^{-1}\Big(\frac{t}{(8\pi r)^{d-1}}\Big)}\,dt\right\} \\
&\geq&
4D||f||_{W^{1,1}}
\end{eqnarray*}
for a suitable selected $r$. In the above computations we used concavity of the function $\Phi^{-1}$. This and $\Phi^{-1}(0)=0$ allow us to write $\Phi^{-1}(\alpha x)\leq\alpha\Phi^{-1}(x)$ for $x\geq 0$ and $\alpha\geq 1$. Notice that 
$$\frac{(8\pi)^d}{2V_d}\geq 1.$$
By the above estimate and Lemma \ref{lemacik1} we get
\begin{eqnarray*}
||f||_{B_{\Phi,1}^\Psi}
\geq
||f||_{L_\Phi}+\frac{1}{2}\int_1^\infty\frac{\Psi(t)}{t}\omega_\Phi(f,\frac{1}{2t})\,dt 
\geq
2D||f||_{W^{1,1}}
\end{eqnarray*}
which contradicts \eqref{eq7}.
\end{proof}

\section{Comparison of two norms defining the Besov-Orlicz spaces}
In this section we will prove the following
\begin{theorem}
\label{comparison}
	For $f\in L_\Phi$ we have the estimate: 
	$$||f||_{\widetilde{B_{\Phi,1}^\Psi}}\lesssim ||f||_{B_{\Phi,1}^\Psi}.$$
\end{theorem}

To prove the above theorem, we need another quantity which is, in some sense, intermediate between these two norms. Let 
$$E_\Phi(f,m)=\inf\big\{||f-g||_{L_\Phi}: \mbox{supp}\,\hat{g}\subseteq\{(a,b)\in\mathbb{Z}^2: |a|, |b|\leq m\}\big\},$$
and
$$||f||_{\overline{B_{\Phi,1}^\Psi}}=\left(1+\Psi(1)+\Psi(2)\right)||f||_{L_\Phi}+\sum_{n=2}^\infty \Psi(2^n) E_\Phi(f,2^{n-3}).$$

Now we are ready to prove the appropriate estimates.
\begin{lemma}
For every $f\in L_\Phi$ we have
$$||f||_{\widetilde{B_{\Phi,1}^\Psi}}\lesssim||f||_{\overline{B_{\Phi,1}^\Psi}}.$$
\end{lemma}
\begin{proof}
We will show that for $n=2, 3, 4, ...$,
$$||g_n*f||_{L_\Phi}\lesssim E_\Phi (f,2^{n-3}).$$
Take a function $h$ such that 
$$\mbox{supp}\,\hat{h}\subseteq\big\{(a,b)\in\mathbb{Z}^2:|a|, |b|\leq 2^{n-3}\big\}.$$
Then for $n\geq 1$ we have $g_n*f=g_n*(f-h)$, so by the Young convolution inequality (see \cite[Theorem 9, page 64]{RaRe})
$$
	||g_n*f||_{L_\Phi}
	=
	||g_n*(f-h)||_{L_\Phi}
	\leq
	2||g_n||_{L_1}\cdot||f-h||_{L_\Phi}
	\leq
	36||f-h||_{L_\Phi}.
$$
Taking the infimum over such functions $h$ we get
$$||g_n*f||_{L_\Phi}\leq 36 E_\Phi(f,2^{n-3}).$$
\end{proof}
The above proof is the same as the proof of similar fact for ordinary Besov spaces, see \cite[Theorem 11, pages 72 - 74]{Pe}.

\begin{lemma}
For every $f\in L_\Phi$ we have
$$||f||_{\overline{B_{\Phi,1}^\Psi}}\lesssim ||f||_{B_{\Phi,1}^\Psi}.$$
\end{lemma}
The proof of this lemma is almost the same as the proof of the Preposition 3.1 in paper \cite[page 88]{PeWoj}. For the reader's convenience we include it here.
\begin{proof}
It is enough to show that
$$E_\Phi(f,2^{n-3})\lesssim \omega_\Phi\left(f,\frac{1}{2^{n-3}}\right).$$
Let $\psi:\mathbb{R}^2\rightarrow\mathbb{C}$ be a function such that 
$$\widehat{\psi}(x,y)=\eta(x,y)\cdot\exp(-|(x,y)|^2),$$
where the smooth function $\eta:\mathbb{R}^2\rightarrow\mathbb{C}$ fulfills $\mbox{supp}\,\eta\subseteq\{(x,y)\in\mathbb{R}^2: |x|, |y|\leq1\}$ and $\eta(0,0)=1$.

Define a family of trigonometric polynomials $\{\phi_n\}_{n=0}^\infty$, $\phi_n:\mathbb{T}^2\rightarrow\mathbb{C}$ by the formula
$$\widehat{\phi_n}(k,l)=\widehat{\psi}\left(\frac{k}{n},\frac{l}{n}\right),$$
Notice that there exist a constant $C>0$ such that
\begin{enumerate}
\item $\int_{\mathbb{T}^2}\phi_n(x)\,dx=\widehat{\phi_n}(0,0)=\widehat{\psi}(0,0),$ for $n=0, 1, ...$;
\item $\mbox{supp}\,\widehat{\phi_n}\subseteq\{(k,l)\in\mathbb{Z}^2: |k|, |l|\leq n\}$;
\item for sufficiently big $m$ and all real $x$ we have: $$\left|\sum_{k,l=-m}^m\widehat{\psi}(k/m,l/m)e^{i(k/m,l/m)\cdot x}\frac{1}{m^2}\right|
\leq
||\eta||_{L_\infty}\frac{(2m+1)^2}{m^2}e^{-|x|^2}
\leq
Ce^{-|x|^2}.$$
\end{enumerate}
Condition 3. above is a consequence of the formula for the inverse Fourier transform and the uniform convergence of the sequence of Riemman sums to the integral - the inverse Fourier transform.

By the Minkowski integral inequality for Orlicz spaces, for $m\in\mathbb{N}$,
\begin{equation}\label{A}
  \begin{aligned}
E_\Phi(f,m)&\leq ||f-\phi_m*f||_{L_\Phi} \\
&=
\left\Vert\int_{\mathbb{T}^2}\phi_m(h)[f(\cdot)-f(\cdot-h)]\,dh\right\Vert_{L_\Phi} \\
&\leq
2\int_{\mathbb{T}^2}|\phi_m(h)| ||f(\cdot)-f(\cdot-h)||_{L_\Phi}\,dh \\
&\leq
2\int_{\mathbb{T}^2}|\phi_m(h)|\omega_\Phi(f,|h|)\,dh.
 \end{aligned}
\end{equation}
Now we use the inequality
\begin{eqnarray}\label{B}
\omega_\Phi(f,|h|)\leq(|h|m+1)\omega_\Phi(f,1/m),
\end{eqnarray}
which is obvious for $|h|<1/m$ and for $|h|>1/m$ it is a consequence of the following reasoning. Let $k$ be a natural number such that $|h|m-1<k\leq|h|m$. Then
$$f(x+h)-f(x)=\sum_{j=1}^k \Big[f\Big(x+\frac{jh}{|h|m}\Big)-f\Big(x+\frac{(j-1)h}{|h|m}\Big)\Big]+f(x+h)-f\Big(x+\frac{kh}{|h|m}\Big),$$
hence
$$\omega_\Phi(f,|h|)\leq(k+1)\omega_\Phi(f,1/m)\leq(|h|m+1)\omega_\Phi(f,1/m).$$
By \eqref{A} and \eqref{B}
$$E_\Phi(f,m)\leq 2\omega_\Phi(f,1/m)\int_{\mathbb{T}^2}(|h|m+1)|\phi_m(h)|\,dh.$$
We get
\begin{eqnarray*}
\int_{\mathbb{T}^2}(|h|m+1)|\phi_m(h)|\,dh&=&
\int_{\mathbb{T}^2}(|h|m+1)\Big|\sum_{k,l=-m}^m\widehat{\phi_m}(k,l)e^{i(k,l)\cdot h}\Big|\,dh \\
&=&
\int_{\mathbb{T}^2}(|h|m+1)\Big|\sum_{k,l=-m}^m\widehat{\psi}(k/m,l/m)e^{i(k,l)\cdot h}\Big|\,dh \\
&=&
\int_{m\mathbb{T}^2}(|x|+1)\Big|\sum_{k,l=-m}^m\widehat{\psi}(k/m,l/m)e^{i(k/m,l/m)\cdot x}\frac{1}{m^2}\Big|\,dx \\
&\leq&
C \int_{m\mathbb{T}^2}(|x|+1)e^{-|x|^2}\,dx \\
&<&
+\infty,
\end{eqnarray*}
by condition 3. above for big enough values of $m$.
\end{proof}

This yields $E_\Phi(f,m)\leq 2K\omega_\Phi(f,1/m)$, for some constant $K>0$. Consequently
\begin{eqnarray*}
||f||_{\widetilde{B_{\Phi,1}^\Psi}}&=&||f||_{L_\Phi}+\sum_{n=0}^\infty \Psi(2^n)||g_n*f||_{L_\Phi} \\
&\lesssim&
\left(1+\Psi(1)+\Psi(2)\right)||f||_{L_\Phi}+\sum_{n=2}^\infty \Psi(2^n) E_\Phi(f,2^{n-3}) \\
&\leq&
\left(1+\Psi(1)+\Psi(2)\right)||f||_{L_\Phi}+\sum_{n=2}^\infty \Psi(2^n)\omega_\Phi\left(f,\frac{8}{2^n}\right)
\end{eqnarray*}

\section{Marcinkiewicz type sampling theorem}

In this section we will generalise one of the results of J\'{o}zef Marcinkiewicz.

\begin{proof}[Proof of Theorem \ref{podmultiplikatywnosc}]
In the book  \cite[Volume II,  Chapter X, Inequality (7.8), page 29]{Zyg} it is proved that for every trigonometric polynomial $g:\mathbb{T}\rightarrow\mathbb{C}$ of order $n$, 
$$g(x)=\sum_{k=-n}^n a_k e^{i k x},$$
where $a_{-n}, a_{-n+1}, ..., a_{n-1}, a_n$ are some complex numbers, and for any non-decreasing, convex function $\Phi:\mathbb{R}_+\rightarrow\mathbb{R}_+$ we have
\begin{eqnarray}\label{nier_Zygmund}
	\frac{1}{2n+1}\sum_{k=-n}^{n} \Phi\left(\frac{1}{3}\left|g\left(e^{2\pi i (k+n)/(2n+1)}\right)\right|\right)\leq\frac{1}{2\pi}\int_{-\pi}^\pi\Phi(|g(x)|)\,dx.
\end{eqnarray}

So, if we assume that the function $g:\mathbb{T}^2\rightarrow\mathbb{C}$ is a trigonometric polynomial such that $\{(k,l)\in\mathbb{Z}^2: \widehat{g}(k,l)\neq 0\}\subseteq G_n\subseteq (-2^n,2^n)^2$, we obtain using inequality \eqref{nier_Zygmund} twice
	\begin{eqnarray*}
	\frac{1}{(2^{n+1}-1)^2}\sum_{k=-2^n+1}^{2^n-1}\sum_{l=-2^n+1}^{2^n-1}
		\Phi\left(\frac{|g(x_k,y_l)|}{9}\right) 
	 &\leq& 
	\frac{1}{2^{n+1}-1}\sum_{k=-2^n+1}^{2^n-1}\frac{1}{2\pi} 
		\int_0^{2\pi} \Phi\left(\frac{|g(x_k,y)|}{3}\right)\,dy \\ &=& 
	\frac{1}{2\pi}\int_0^{2\pi} \frac{1}{2^{n+1}-1}\sum_{k=-2^n+1}^{2^n-1}
		\Phi\left(\frac{|g(x_k,y)|}{3}\right)\,dy \\ &\leq& 
	\frac{1}{(2\pi)^2}\int_0^{2\pi}\int_0^{2\pi}\Phi\big(|g(x,y)|\big)\,dx\,dy.
	\end{eqnarray*}

Without loss of generality we assume that $||g||_{L_\Phi}=1$.
Then
\begin{eqnarray}\label{nier_trywialna}
\frac{1}{w_n}\sum_{(k,l)\in G_n} \Phi\left(\frac{|g(x_k,y_l)|}{9}\right) 
&\leq&
\frac{(2^{n+1}-1)^2}{w_n}\frac{1}{(2^{n+1}-1)^2}\sum_{k=-2^n+1}^{2^n-1}\sum_{l=-2^n+1}^{2^n-1}
\Phi\left(\frac{|g(x_k,y_l)|}{9}\right) \nonumber\\ 
&\leq&
\frac{4}{3}\frac{1}{(2\pi)^2}\int_{\mathbb{T}^2}\Phi\big(|g(x,y)|\big)\,dx\,dy 
\leq
\frac{4}{3}. 
\end{eqnarray}
Let
$$B_n=\left\{(k,l)\in G_n: \frac{|g(x_k,y_l)|}{12}\geq C\right\},$$
and
$$S_n=G_n\setminus B_n.$$

For the set $B_n$ we will use the assumption: $\Phi(a)\Phi(b)\leq\Phi(Cab)$ for $0<a<1\leq ab<b$. We can rewrite this condition in the form $\Phi\big(\frac{p}{Cq}\big)\leq\frac{\Phi(p)}{\Phi(q)}$ where $p=Cab, q=b$ and
\begin{eqnarray}\label{nier_mult}
\frac{p}{Cq}<1\leq \frac{p}{C}<q.
\end{eqnarray}

Let us set
$$p=\frac{3|g(x_k,y_l)|}{4\cdot 9} \mbox{ and } q=\Phi^{-1}(w_n).$$
Such $p$ and $q$ satisfies inequality \eqref{nier_mult}. We prove this fact at the end of the proof.
By condition 1 of theorem (super-multiplicative property), convexity of $\Phi$ and \eqref{nier_trywialna},
$$
\sum_{(k,l)\in B_n} \Phi\Big(\frac{3|g(x_k,y_l)|}{C4\cdot 9\Phi^{-1}(w_n)}\Big) 
\leq
\frac{1}{w_n}\sum_{(k,l)\in B_n}
\Phi\Big(\frac{3|g(x_k,y_l)|}{4\cdot 9}\Big) 
\leq
\frac{3}{4w_n}\sum_{(k,l)\in G_n}
\Phi\Big(\frac{|g(x_k,y_l)|}{9}\Big) 
\leq
1.
$$
Therefore
\begin{eqnarray}\label{Bn}
||((g)\chi_{B_n})||_{\ell_\Phi}
=
\inf\Big\{\lambda>0: \sum_{(k,l)\in B_n} \Phi\Big(\frac{|g(x_k,y_l)|}{\lambda}\Big)\leq 1\Big\}
\leq
12C\Phi^{-1}(w_n)||g||_{L_\Phi}.
\end{eqnarray}

Now we estimate the Luxemburg norm of the function $g$ restricted to the set $S_n$.
\begin{eqnarray*}
||((g)\chi_{S_n})||_{\ell_\Phi}
&=&
\inf\Big\{\lambda>0: \sum_{(k,l)\in S_n} \Phi\Big(\frac{|g(x_k,y_l)|}{\lambda}\Big)\leq 1\Big\} \\
&\leq&
\inf\Big\{\lambda>0: \sum_{(k,l)\in S_n} \Phi\Big(\frac{12C}{\lambda}\Big)\leq 1\Big\} \\
&\leq&
\inf\Big\{\lambda>0: w_n\Phi\Big(\frac{12C}{\lambda}\Big)\leq 1\Big\} \\
&\leq&
\inf\Big\{\lambda>0: \frac{12C}{\Phi^{-1}\Big(\frac{1}{w_n}\Big)}\leq \lambda\Big\} \\
&=&
12C\frac{1}{\Phi^{-1}\Big(\frac{1}{w_n}\Big)} \\
&\leq&
12C^2\Phi^{-1}(w_n) ||g||_{L_\Phi}. 
\end{eqnarray*}
We used the fact that $1\leq C\Phi^{-1}(1/x)\cdot\Phi^{-1}(x)$ for $x>0$, and the fact that $||g||_{L_\Phi}=1$.

By the above calculations and \eqref{Bn}
\begin{eqnarray*}
||(g)||_{\ell_\Phi}
&\leq&
||((g)\chi_{B_n})||_{\ell_\Phi}+||((g)\chi_{S_n})||_{\ell_\Phi} \\
&\leq&
12C\Phi^{-1}(w_n)||g||_{L_\Phi}+12C^2\Phi^{-1}(w_n) ||g||_{L_\Phi} \\
&=&
24C^2\Phi^{-1}(w_n)||g||_{L_\Phi}.
\end{eqnarray*}

It remains to show the inequality \eqref{nier_mult}. It is enough to prove
$$\frac{|g(x_k,y_l)|}{9\Phi^{-1}(w_n)}<4/3\mbox{, for } (k,l)\in G_n.$$
Remembering that $||g||_{L_\Phi}=1$ and relying on the estimate \eqref{nier_trywialna}, we can write that
$$
\frac{1}{w_n}\Phi\Big(\frac{|g(x_k,y_l)|}{9}\Big)
\leq
\frac{1}{w_n}\sum_{(k,l)\in G_n}\Phi\Big(\frac{|g(x_k,y_l)|}{9}\Big) 
\leq
\frac{4}{3}.
$$
This inequality can be rewritten in the form:
$$|g(x_k,y_l)|\leq 9\Phi^{-1}(4/3\cdot w_n)\leq 12\Phi^{-1}(w_n).$$
Here, we used the concavity of the function $\Phi^{-1}$, which implies that $\Phi^{-1}(4/3\cdot w_n)\leq 4/3\cdot\Phi^{-1}(w_n)$.
The proof is complete.
\end{proof}

\section{Factoring through embedding $\ell^\Phi\hookrightarrow\ell^2$}

In this section we will prove Theorem \ref{glowne_twierdzenie}. The key ingredient will be the following Theorem \ref{summing_embedding}, which stands behind estimates of operator $C$ from diagram \eqref{diagram}.
\begin{theorem}
\label{summing_embedding}
The embedding of the space $\widetilde{B_{\Phi,1}^\Psi}$ into the space $\widetilde{B_{2,2}^0}$ factorizes through embedding $\ell_\Phi\hookrightarrow\ell_2$ provided that the function $\Phi$ fulfils the following conditions:
\begin{enumerate}[(i)]
\item $\frac{1}{t}\leq \frac{\Psi(t)}{\Phi^{-1}(t^2)}$ for every $t\in\mathbb{R}_+$,
\item the function $x\mapsto \Phi(\sqrt{x})$, $x\in\mathbb{R}_+$, is concave,
\item $\Phi(a)\Phi(b)\leq\Phi(Cab)$ for all positive real numbers $0<a<1\leq ab<b$ and some $C>0$ (restricted supermultiplicativity),
\item $1\leq C\Phi^{-1}(x)\Phi^{-1}(1/x)$ for all positive $x$ and some $C>0$.
\end{enumerate}
\end{theorem}

For the proof of Theorem \ref{summing_embedding} we need one more estimate which concerns the Hilbert spaces.
\begin{lemma}\label{Marcinkiewicz_l2}
For every $n=0, 1, 2, ...$ and $f\in L_2$ we have the inequality
$$||g_n*f||_{L_2}\lesssim \frac{1}{\sqrt{w_n}}||(g_n*f)||_{\ell_2}.$$
\end{lemma}

For the definition of the norm $||(\cdot)||_{l_2}$ see \eqref{small_norm}.
The proof of this lemma relies on the Marcinkiewicz sampling theorem and is contained in the paper \cite[Lemma 7, page 167]{Woj}.
\begin{proof}[Proof of Theorem \ref{summing_embedding}.]
We have the following collection of inequalities.
\begin{eqnarray*}
\Big(\sum_{n\in\mathbb{N}}||g_n*f||_{L_2}^2\Big)^{1/2} 
&\stackrel{(1)}{\lesssim}&
\Big(\sum_{n\in\mathbb{N}}\frac{1}{w_n}||(g_n*f)||_{\ell_2}^2\Big)^{1/2} \\
&=&
\Big(\sum_{n\in\mathbb{N}}\sum_{(k,l)\in G_n}\big(\frac{1}{\sqrt{w_n}}|g_n*f(x_k,y_l)|\big)^2\Big)^{1/2} \\
&\stackrel{(2)}{\leq}&
\inf\big\{\lambda>0:\sum_{n\in\mathbb{N}}\sum_{(k,l)\in G_n}\Phi\Big(\frac{\frac{1}{\sqrt{w_n}}|g_n*f(x_k,y_l)|}{\lambda}\Big)\leq 1\big\} \\
&\stackrel{(3)}{\leq}&
\inf\big\{\lambda>0:\sum_{n\in\mathbb{N}}\sum_{(k,l)\in G_n}\Phi\Big(\frac{\Psi(\sqrt{w_n})\frac{1}{\Phi^{-1}(w_n)}|g_n*f(x_k,y_l)|}{\lambda}\Big)\leq 1\big\} \\
&\stackrel{(4)}{\leq}&
\sum_{n\in\mathbb{N}}\inf\big\{\lambda>0\sum_{(k,l)\in G_n}\Phi\Big(\frac{\Psi(\sqrt{w_n})\frac{1}{\Phi^{-1}(w_n)}|g_n*f(x_k,y_l)|}{\lambda}\Big)\leq 1\big\}  \\
&=&
\sum_{n\in\mathbb{N}}\Psi(\sqrt{w_n})\frac{1}{\Phi^{-1}(w_n)}||(g_n*f)||_{\ell_\Phi} \\
&\stackrel{(5)}{\leq}&
24C^2\sum_{n\in\mathbb{N}}\Psi(\sqrt{w_n})\frac{1}{\Phi^{-1}(w_n)}\Phi^{-1}(w_n)||g_n*f||_{L_\Phi} \\
&=&
24C^2\sum_{n\in\mathbb{N}}\Psi(\sqrt{w_n})||g_n*f||_{L_\Phi}
\end{eqnarray*}

Inequality (1) follows from Lemma \ref{Marcinkiewicz_l2}. Inequality (2) is the continuous embedding of the space $\ell_\Phi$ into the space $l_2$ which follows from $(ii)$. Inequality (3) follows from $(i)$ and monotonicity of function $\Phi$. Inequality (4) is the triangle inequality for an infinite sum. Inequality (5) follows from Theorem \ref{podmultiplikatywnosc}. Theorem \ref{summing_embedding} follows from the fact that $\sqrt{w_n}\leq2\cdot2^n$.
\end{proof}

We will now formulate a slightly simpler version of Theorem \ref{summing_embedding}.
\begin{theorem}
\label{small_summing}
The embedding of the space $\widetilde{B_{\Phi,1}^\Psi}$ into the space $\widetilde{B_{2,2}^0}$ factorizes through embedding $\ell_\Phi\hookrightarrow\ell_2$ for the function $\Phi$ that fulfils the following conditions:
\begin{enumerate}[(i)]
\item $\frac{1}{t}\leq \frac{\Psi(t)}{\Phi^{-1}(t^2)}$ for every $t\in\mathbb{R}_+$,
\item the function $x\mapsto \Phi(\sqrt{x})$, $x\in\mathbb{R}_+$, is concave,
\item $\Phi(a)\Phi(b)\leq\Phi(ab)$ for all real numbers $0<a<1\leq ab<b$,
\item $\Phi^{-1}(1)=1$.
\end{enumerate}
\end{theorem}
\begin{proof}
It is enough to notice that the condition $\Phi(a)\Phi(b)\leq\Phi(ab)$, for $0<a<1\leq ab<b$ implies, by the monotonicity of the function $\Phi^{-1}$, the inequality
$$\Phi^{-1}(xy)\leq\Phi^{-1}(x)\Phi^{-1}(y),$$
where $x=\Phi(a)$ and $y=\Phi(b)$, for
$$0<x<1\leq xy<y.$$
By the property of Young functions we can take $y=1/x$ to get condition $(iv)$ of Theorem \ref{summing_embedding}.
\end{proof}

Now we can turn to the proof of Theorem \ref{glowne_twierdzenie}.
\begin{proof}[Proof of Theorem \ref{glowne_twierdzenie}]

The steps of the proof are described by diagram \eqref{diagram}. 
Let us set
\begin{equation}
\label{Phi}
\Psi(t)=\frac{\Phi^{-1}(t^2)}{t}.    
\end{equation}
Then condition $1.$ form Theorem \ref{glowne_twierdzenie} takes the form of condition \eqref{Kolyada_nier} from Theorem \ref{Kolyada}. This means that the embedding denoted by arrow $A$ in diagram \eqref{diagram} is a bounded operator. Arrow $B$ on the same diagram describe bounded operator by Theorem \ref{comparison}. Boundedness of operator $C$ in diagram \eqref{diagram} is a consequence of Theorem \ref{summing_embedding}. Assumptions $(ii)$ to $(iv)$ of this theorem are the same as assumptions $2.$, $3.$ and $4.$ of Theorem \ref{glowne_twierdzenie}. Moreover, assumption $(i)$ of Theorem \ref{summing_embedding} is satisfies because of equation \eqref{Phi}. Theorem \ref{summing_embedding} also guarantee that the embedding denoted by $C$ in diagram \eqref{diagram} factorizes through embedding $\ell_\Phi\hookrightarrow\ell_2$ and thus the Sobolev embedding operator also factorizes this way.
\end{proof}

\section{An example of Orlicz function}
We will construct a functions which satisfies the assumptions of Theorem \ref{glowne_twierdzenie}. Let
$$\Phi^{-1}(t)=
\begin{cases} 
t\cdot e^{\alpha\frac{\ln(1/\sqrt{t})}{\ln\ln(1/\sqrt{t})}}
&\text{ dla } t\in[0,1/r) \\
pt+q
&\text{ dla } t\in[1/r,r) \\
t\cdot e^{-\alpha\frac{\ln\sqrt{t}}{\ln\ln\sqrt{t}}}
&\text{ dla } t\geq r
\end{cases}.$$
We define 
$$r=e^{2e^2}$$
and we want from $\alpha$ to be smaller than $e^{-2}$. Then
$$\alpha\frac{\ln r}{\ln\ln r}\leq 1.$$
Moreover we choose $p$ and $q$ in such a way that the function $\Phi^{-1}$ is continuous. So, it is not hard to see that 
\begin{eqnarray*}
p=\frac{re^{-\alpha\frac{\ln(\sqrt{r})}{\ln\ln(\sqrt{r})}}-\frac{1}{r}e^{\alpha\frac{\ln(\sqrt{r})}{\ln\ln(\sqrt{r})}}}{r-1/r}
\mbox{ and }
q=re^{-\alpha\frac{\ln(\sqrt{r})}{\ln\ln(\sqrt{r})}}-pr.
\end{eqnarray*}

Notice that we have the following estimates
\begin{eqnarray}\label{pq}
\frac{1}{2e}\leq\frac{re^{-\alpha\frac{\ln(\sqrt{r})}{\ln\ln(\sqrt{r})}}}{2r}<p<\frac{r-1/r}{r-1/r}=1
\mbox{ and }
0<q\leq\frac{10}{r}p.
\end{eqnarray}
We also have the estimate
$$t^{1-\frac{\alpha}{2}}=t\cdot e^{-\frac{\alpha}{2}\ln t}\leq t\cdot e^{-\alpha\frac{\ln\sqrt{t}}{\ln\ln\sqrt{t}}}\leq t\cdot e^{-\alpha}$$
for $t\geq e^{2e}$.

First notice that condition 4 of Theorem \ref{glowne_twierdzenie} is clearly fulfilled:
$$\Phi^{-1}(x)\cdot \Phi^{-1}(1/x)=1,$$
for all sufficiently large positive values $x$.

Now we will check the integral condition 1 from Theorem \ref{glowne_twierdzenie}, that is we will show the existence of the constant $C>0$ such that for every $s\geq 1$ we have
$$\frac{s}{\Phi^{-1}(s^2)}\int_1^s\frac{\Phi^{-1}(t^2)}{t^2}\, dt+\int_s^\infty\frac{\Phi^{-1}(t^2)s}{t^2\Phi^{-1}(ts)}\,dt<C.$$
We will consider two cases:
\begin{enumerate}
\item $1\leq s\leq r$,
\item $r<s$.
\end{enumerate}

In case 1, by \eqref{pq} we get
\begin{eqnarray*}
\frac{s}{\Phi^{-1}(s^2)}\int_1^s\frac{\Phi^{-1}(t^2)}{t^2}\, dt
&=&
\frac{s}{ps^2+q}\int_1^s\frac{pt^2+q}{t^2}\,dt \\
&=&
\frac{s}{ps^2+q}\left(ps-\frac{q}{s}-p+q\right) \\
&\leq&
1+\frac{q(s-2)}{ps^2+q} \\
&\leq&
2
\end{eqnarray*}
and for the second integral we will use the substitutions $\ln t=x$ and $\ln r=k$. We then get
\begin{eqnarray*}
\int_s^\infty\frac{\Phi^{-1}(t^2)s}{t^2\Phi^{-1}(ts)}\,dt
&=&
\int_s^r\frac{\Phi^{-1}(t^2)s}{t^2\Phi^{-1}(ts)}\,dt+\int_r^\infty\frac{\Phi^{-1}(t^2)s}{t^2\Phi^{-1}(ts)}\,dt \\
&=&
\int_s^r\frac{(pt^2+q)s}{t^2(pts+q)}\,dt+\int_r^\infty\frac{1}{t}e^{\alpha\big(\frac{\ln\sqrt{st}}{\ln\ln\sqrt{st}}-\frac{\ln t}{\ln\ln t}\big)}\,dt \\ 
&\leq&
\int_s^r\frac{(pts+q)t}{t^2(pts+q)}\,dt+\int_k^\infty e^{\alpha\big(\frac{1/2(x+k)}{\ln(x+k)-\ln 2}-\frac{x}{\ln x}\big)}\,dx \\
&\leq&
\ln{r}+\int_k^\infty e^{\alpha\cdot\frac{x\ln x+k\ln x-2x\ln x-2x\ln 2}{2(\ln x-\ln 2)\ln x}}\,dx \\
&=&
2e^2+\int_k^\infty e^{\alpha\cdot\frac{(k-x)\ln x-2x\ln 2}{2(\ln x-\ln 2)\ln x}}\,dx \\
&\leq&
2e^2+\int_k^\infty e^{\alpha\cdot\frac{-2x\ln 2}{2(\ln x-\ln 2)\ln x}}\,dx \\
&<&
\infty
\end{eqnarray*}

Now we will deal with the second case, $r<s$. Let us define $\beta$ to be equal $\frac{\alpha}{\ln\ln s}$. Then we have 
$$e^{-\alpha\cdot\frac{\ln t}{\ln\ln t}}\leq t^{-\beta},$$
for $t<s$. By \eqref{pq} we can calculate
\begin{eqnarray*}
\frac{s}{\Phi^{-1}(s^2)}\int_1^s\frac{\Phi^{-1}(t^2)}{t^2}\, dt
&=&
\frac{s}{\Phi^{-1}(s^2)}\int_1^r\frac{\Phi^{-1}(t^2)}{t^2}\, dt+\frac{s}{\Phi^{-1}(s^2)}\int_r^s\frac{\Phi^{-1}(t^2)}{t^2}\, dt \\
&\leq&
\frac{s}{ps^2+q}\int_1^s\frac{pt^2+q}{t^2}\,dt+\frac{1}{s}e^{\alpha\frac{\ln s}{\ln\ln s}}\int_r^s e^{-\alpha\frac{\ln t}{\ln\ln t}}\,dt \\
&\leq&
2+e^{(\beta-1)\cdot e^{\alpha/\beta}}\int_r^{e^{e^{\alpha/\beta}}} e^{-\beta\ln t}\,dt \\
&=&
2+e^{(\beta-1)\cdot e^{\alpha/\beta}}\cdot\frac{1}{1-\beta}\cdot\Big(e^{(1-\beta)\cdot e^{\alpha/\beta}}-r^{1-\beta}\Big) \\
&\leq&
2+\frac{1}{1-\beta} \\
&\leq&
2+\frac{1}{1-\frac{\alpha}{\ln\ln r}}.
\end{eqnarray*}

Moreover, as we show in case 1, the estimate
$$\int_s^\infty\frac{\Phi^{-1}(t^2)s}{t^2\Phi^{-1}(ts)}\,dt\leq\int_r^\infty\frac{\Phi^{-1}(t^2)s}{t^2\Phi^{-1}(ts)}\,dt<\infty$$
holds.

For condition 2 of Theorem \ref{glowne_twierdzenie} notice that concavity of the function $x\mapsto\Phi(\sqrt{x})$ is equivalent to the statement that function $y\mapsto\big(\Phi^{-1}(x)\big)^2$ is convex. Let us put $f:[0,\infty)\rightarrow [0,\infty)$, $f(x)=\left(\Phi^{-1}(x)\right)^2$ and $y(x)=\ln(1/\sqrt{x})$. We have for $x\in(0,1/r)$
$$\frac{d^2}{dx^2}f(x)=e^{2\alpha\frac{y(x)}{\ln(y(x))}}
\left[2+\alpha\left(-3\frac{\ln(y(x))-1}{(\ln(y(x)))^2}+\frac{2-\ln(y(x))}{2(\ln(y(x)))^3y(x)}+\alpha\frac{(\ln(y(x))-1)^2}{(\ln(y(x)))^4}\right)\right],$$
for $x\in(1/r,r)$ we have
$$\frac{d^2}{dx^2}f(x)=2p^2$$
and for $x\in(r,\infty)$ we have
$$\frac{d^2}{dx^2}f(x)=e^{-2\alpha\frac{\ln{\sqrt{x}}}{\ln{\ln{\sqrt{x}}}}}
\left[2+\alpha\left(-3\frac{\ln{\ln{\sqrt{x}}}-1}{(\ln{\ln{\sqrt{x}}})^2}+\frac{\ln{\ln{\sqrt{x}}}-2}{2(\ln{\ln{\sqrt{x}}})^3\ln{\sqrt{x}}}+\alpha\frac{(\ln{\ln{\sqrt{x}}}-1)^2}{(\ln{\ln{\sqrt{x}}})^4}\right)\right],$$
So we see that it is enough to choose $\alpha>0$ sufficiently small to guaranty positivity of second derivative of the function $f$.

Now we can check the sup-multiplicative condition 3 of Theorem \ref{glowne_twierdzenie} for function $\Phi$: there is constant $C$ such that $\Phi(Cab)\geq \Phi(a)\cdot\Phi(b)$, for all $0<a<1\leq ab<b$. Before we go into further considerations, we need one more simple observation which will show that sup-multiplicativity of the function $\Phi$ is equivalent to sub-multiplicativity of the function $\Phi^{-1}$.
We will use the following lemma:
\begin{lemma}\label{multiplikatywnosc}
Let $\Phi$ be a Young function. Moreover let $0<x<\Phi(1)<y$, $1\leq\Phi^{-1}(x)\Phi^{-1}(y)$ and $C\geq 1$ be some constant. Then 
$$\Phi^{-1}(xy)\leq C\Phi^{-1}(x)\Phi^{-1}(y)$$
implies that
$$\Phi(a)\Phi(b)\leq\Phi(Cab),$$
where $a$ and $b$ are positive numbers such that $\Phi^{-1}(x)=a$ and $\Phi^{-1}(y)=b$ and
$$0<a<1\leq ab<b.$$ 
\end{lemma}
\begin{proof}
Function $\Phi$ is increasing, so if we act this functions on both sides of the inequality in the assumption, then we get
$$xy\leq \Phi\big(C\Phi^{-1}(x)\Phi^{-1}(y)\big).$$
Now we use the assumption, that $\Phi^{-1}(x)=a$ and $\Phi^{-1}(y)=b$, to get
$$\Phi(a)\Phi(b)\leq\Phi(Cab).$$
The condition $0<a<1\leq ab<b$ are consequences of the assumptions.
\end{proof}

Now we will try to find a constant $C>0$ such that
$$\Phi^{-1}(xy)\leq C\Phi^{-1}(x)\Phi^{-1}(y)$$
for  $0<x<\Phi(1)<y$, $1\leq\Phi^{-1}(x)\Phi^{-1}(y)$ and $C\geq 1$.

We will consider three cases:
\begin{enumerate}
\item $xy\leq r$,
\item $0<x<1/r<r<xy<y$,
\item $1/r<x<r<xy<y$,
\end{enumerate}
Now we will check all these cases. 

1. For this case we can see that
$$\Phi^{-1}(xy)\leq\Phi^{-1}(r)\leq r\leq r\Phi^{-1}(x)\Phi^{-1}(y).$$

2. Our condition for the above function takes the form:
$$xye^{-\alpha\frac{\ln{\sqrt{xy}}}{\ln\ln\sqrt{xy}}}\leq xe^{\alpha\frac{\ln(1/\sqrt{x})}{\ln\ln(1/\sqrt{x})}} ye^{-\alpha\frac{\ln\sqrt{y}}{\ln\ln\sqrt{y}}}.$$
After small manipulations we get the condition:
$$\frac{\ln\sqrt{xy}}{\ln\ln\sqrt{xy}}\geq-\frac{\ln(1/\sqrt{x})}{\ln\ln(1/\sqrt{x})}+\frac{\ln\sqrt{y}}{\ln\ln\sqrt{y}}.$$
Now we use substitutions $a=\ln\sqrt{x}$ and $b=\ln\sqrt{y}$. Notice that if $r\geq e^{2e^2}$, then we have $a<-e^2$, $b>e^2$ and $a+b>e^2$. So we get
$$\frac{a+b}{\ln(a+b)}\geq -\frac{-a}{\ln(-a)}+\frac{b}{\ln b}.$$
For convenience we assume that $c=-a>e^2$. Then our equation has the form
$$\frac{b-c}{\ln(b-c)}\geq-\frac{c}{\ln c}+\frac{b}{\ln b},$$
in other form
$$\frac{b-c}{\ln(b-c)}+\frac{c}{\ln c}\geq\frac{b}{\ln b},$$
where $b,c,b-c>e^2$. Now notice that for $t>e^2$ function $f(t)=\frac{t}{\ln t}$ in concave. We can also redefine it on the positive neighbourhood of zero in such a way that it will be concave on the whole half-line $t\geq0$ and satisfies the condition $f(0)\geq0$. Then this function, which we will also denote by $f$, would satisfies the condition $f(u)+f(v)\geq f(u+v)$, where $u$ and $v$ are two positive numbers. This condition is equivalent to
$$f(b-c)+f(c)\geq f(b).$$
So we can write
$$\frac{b-c}{\ln(b-c)}+\frac{c}{\ln c}\geq\frac{b}{\ln b}.$$

3. The last case. We will prove that
$$xye^{-\alpha\frac{\ln(\sqrt{xy})}{\ln\ln(\sqrt{xy})}}\leq\frac{e^{\alpha\frac{\ln(\sqrt{r})}{\ln\ln(\sqrt{r})}}}{p}(px+q)ye^{-\alpha\frac{\ln(\sqrt{y})}{\ln\ln(\sqrt{y})}}$$
for every $1/r\leq x<r<xy<y$. Using an estimation similar to the one from point 2 and an inequality $\frac{1}{x}\leq r$ we get
$$xye^{-\alpha\frac{\ln(\sqrt{xy})}{\ln\ln(\sqrt{xy})}}\leq e^{\alpha\big(\frac{\ln(\sqrt{y})}{\ln\ln(\sqrt{y})}-\frac{\ln(\sqrt{xy})}{\ln\ln(\sqrt{xy})}\big)}\frac{px+q}{p}ye^{-\alpha\frac{\ln(\sqrt{y})}{\ln\ln(\sqrt{y})}}\leq \frac{e^{\alpha\frac{\ln(\sqrt{r})}{\ln\ln(\sqrt{r})}}}{p}(px+q)ye^{-\alpha\frac{\ln(\sqrt{y})}{\ln\ln(\sqrt{y})}}.$$

From the above calculations we see that the inequality
$$\Phi^{-1}(x\cdot y)\leq C\Phi^{-1}(x)\cdot\Phi^{-1}(y),$$
is true for all $x, y$ such that $0<x<\Phi(1)<y$, $1\leq\Phi^{-1}(x)\Phi^{-1}(y)$, where C is the largest of the values
$$r, 1, \frac{e^{\alpha\frac{\ln(\sqrt{r})}{\ln\ln(\sqrt{r})}}}{p}.$$
The last value can be estimate from above by
$$\frac{e}{p}\leq 2e^2$$
because $\frac{1}{2e}<p$ and $\alpha\frac{\ln r}{\ln\ln r}\leq 1$.
So $C=r=e^{2e^2}$ will be good.

\section{$(\Phi,1)$-summing embedding via extrapolation}

In this section we will prove Theorem \ref{invariantnafaktor}. The proof uses estimates obtained in \cite{Woj} and extrapolation techniques.

We denote by $L^1_{w} (I)$ the space of functions integrable with weight $w$ on the interval $I$.
One of the key ingredients in the proof of Theorem \ref{invariantnafaktor} is a simple extrapolation lemma for sequence spaces inspired by the Theorem of Yano \cite{Yano}.
\begin{lemma}
\label{one}
Let $\|x\|^p_{\ell^p}\leq f(p)$ for $q+\varepsilon>p>q\geq 1$ and $f \in L^1_{w} ([q,q+\varepsilon])$, where $w(p)=(p-q)^{\alpha}$ and $\alpha>-1$. Then $x\in \ell^{\Phi}$ where $\Phi(x)=\frac{x^q}{|\ln(x)|^{\alpha+1}}$ for $0 \leq x \leq \frac{1}{2}$.
\end{lemma}

\begin{proof}
We normalize the sequence $x$ in such a way that $f(q+\varepsilon)=\frac{1}{2}$. Hence $x_k\leq \frac{1}{2}$ for any $k$. We define $K_n=\{x_k:  \frac{1}{n}< x_k\leq\frac{1}{n-1}\}$, for $n=3,4, ...$. Observe that the following inequality is satisfied
\[
\sum_{n=3}^\infty \# K_n \frac{1}{n^p} \leq \|x\|^p_{\ell^p} \leq f(p) .
\]
Hence 
\[
\sum_{n=3}^\infty \# K_n \frac{1}{n^p} (p-q)^{\alpha}\leq f(p)(p-q)^{\alpha}.
\]
We integrate the inequality on the interval $(q,q+\varepsilon)$
\[
\sum_{n=3}^\infty \# K_n \frac{1}{n^q} \int_{q}^{q+\varepsilon} \frac{1}{n^{p-q}} (p-q)^{\alpha}\,dp\leq \int_{q}^{q+\varepsilon}f(p)(p-q)^{\alpha}dp< +\infty.
\]
We estimate the left hand side of the inequality 
\[
\begin{split}
\sum_{n=3}^\infty \# K_n \frac{1}{n^q} \int_{q}^{q+\varepsilon} \frac{1}{n^{p-q}} (p-q)^{\alpha}\,dp
&= \sum_{n=2}^\infty \# K_n \frac{1}{n^q} \int_0^{\varepsilon} e^{-t \ln(n)} t^{\alpha} dt
\\&= \sum_{n=3}^\infty \# K_n \frac{1}{n^q\ln^{\alpha}(n)} \int_0^{\varepsilon} e^{-t \ln(n)} (t\ln(n))^{\alpha} dt
\\&=\sum_{n=3}^\infty \# K_n \frac{1}{n^q\ln^{\alpha+1}(n)} \int_0^{\varepsilon \ln(n)} e^{-s} s^{\alpha} ds 
\\&\geq \left(\int_{0}^{\varepsilon\ln(2)} e^{-s} s^{\alpha}ds\right) \sum_{n=3}^\infty\# K_n \frac{1}{n^q\ln^{\alpha+1}(n)}
\\ &\geq C \sum \frac{(x_k)^q}{|\ln(x_k)|^{\alpha+1}.} 
\end{split} 
\]
Therefore $x\in \ell_{\Phi}$.
\end{proof}
As a consequence we obtain:
\begin{theorem}\label{extrapol}
Let operator $T\in \Pi_{v,1}$ for $v\in (v_1,v_2)$. If there exists $\alpha \geq -1$ such that
\[
\int_{v_1}^{v_2} \left(\pi_{v,1}(T)\right)^{v} (v-v_1)^{\alpha} ds = K
\]
then $T\in \Pi_{\Phi,1}$, where $\Phi(x)=\frac{x^{v_1}}{|\ln(x)|^{\alpha+1}}$ for $0 \leq x \leq \frac{1}{2}$.
\end{theorem}
\begin{proof}
Let $\{x_i\}_{i=1}^{n}$ be sequence in $X$ such that 
\[
\sup\limits_{\|x^*\|_{X^*}= 1}\; \sum_{i=1}^n |x^*(x_i)|=1  
\]
We put $y=(\|T x_1\|_{Y},\|T x_2\|_{Y},\ldots , \|T x_n\|_{Y})$. 
From $(v,1)$-summability of $T$ we get
\[
\|y\|_{\ell_v}\leq \pi_{v,1}(T),
\]
where $v\in (v_1,v_2)$.
Since the function $f(v)=\pi_{v,1}(T)$ is in $L^1_w(v_1,v_2)$, where $w=(v-v_1)^{\alpha}$. From Lemma \ref{one} we get $x\in \ell_{\Phi}$.
\[
\|y\|_{\ell^{\Phi}}\leq C(K,\alpha,|v_1 -v_2|)
\]
The $(\Phi,1)$-summing property follows from the definition of the sequences $y$.
\end{proof}

Similarly as in \cite{Woj} Theorem \ref{invariantnafaktor} is a direct consequence of the following factorization of the Sobolev embedding and the Theorem \ref{extrapol}.
\begin{center}
\begin{tikzpicture}
  \matrix (m) [matrix of math nodes,row sep=3em,column sep=4em,minimum width=2em]
  {
     W^{k,p}(\mathbb{T}^d)& B^{\theta}_{q,p}(\mathbb{T}^d) & B^{\theta}_{q,q}(\mathbb{T}^d) & B^{\lambda}_{r,r}(\mathbb{T}^d)& L_{s}(\mathbb{T})&  \\
  };
  \path[-stealth]
   (m-1-1) edge node [above] {$T_1$} (m-1-2)
    (m-1-2) edge  node [above] {$T_2$} (m-1-3)
      (m-1-3) edge node [above] {$T_3$} (m-1-4)
      (m-1-4)  edge node [above] {$T_4$} (m-1-5)
          ;
\end{tikzpicture}
\end{center}
where $p< q< 2$, $\theta=d (1/q+k/d-1/p)$, $\lambda=d(k/d+1/r-1/p)$ and $r=\min\{s,2\}$.

\begin{proof}[Proof of Theorem \ref{invariantnafaktor}]
To apply Theorem \ref{extrapol} we need a precise bounds on the value of $\pi_{v,1}(S_{d,k,p})$. We briefly reproduce steps of the argument in \cite{Woj}, keeping our attention on the growth of the norm with respect to $v$.\\
We fix $p$,$d$ and $k$. Following carefully the argument from \cite{Woj} and  \cite{PeWoj} we get an upper bound on the norm of the operator $T_1$. Indeed, by \cite{PeWoj} the embedding $Id: W^{1,1} \rightarrow B^{\tau(h,d,1)}_{h,1}$  satisfies the following inequality
\[
\left\|Id: W^{1,1} \rightarrow B^{\tau(h,d)}_{h,1}\right\|\leq  \frac{K}{(h-1)},
\]
where $\tau(h,d,k)=d(\frac{1}{h}+\frac{k}{d}-1).$
By Lemma 3 from  \cite{Woj} we know that the embedding $Id: W^{k,1} \rightarrow B^{\theta(h,d,k)}_{h,1}$ satisfies
\[
\left\|Id: W^{k,1} \rightarrow B^{\tau(h,d,k)}_{h,1}\right\|\leq  \frac{C_1}{(h-1)}.
\]
We use the complex interpolation. Let $\alpha=\frac{2}{p}-1$ and $\frac{1}{q}=\frac{\alpha}{h}+\frac{1-\alpha}{2}$. Since $\left(W^{k,2},W^{k,1}\right)_{\alpha}=W^{k,p}$ and $\left(B^k_{2,2},B^{\tau(h,d,k)}_{h,1}\right)_{\alpha}=B^{\theta}_{q,p}$ and
\[
\left\|Id: W^{k,2} \rightarrow B^{k}_{2,2}\right\|\leq  C_2,
\]
we obtain
\[
\|T_1\|\leq C_3 (q-p)^{1-\frac{2}{p}}.
\]
Since $q>p$ we have $\|T_2\|\leq 1$. Operator $T_3$ by the Marcinkiewicz sampling theorem factorizes uniformly through embedding $\ell_q \rightarrow \ell_r$ (cf. \cite{Woj}, Lemma 7). Indeed, by the Marcinkiewicz sampling theorem \cite[ Volume II, Theorem 7$\cdot$ 5, ($7\cdot6$), pp.~28-29]{Zyg}  
we have the inequality
\begin{equation}\label{eq: sampMarc}
\left\|\left(\phi_j*f\right)\right\|_{\ell_q}\leq  2^{\frac{jd}{q}} A\left\|  \phi_j* f\right\|_{L_q}.  
\end{equation}
Thus,
\[
\|f\|^q_{B^{\theta}_{q,q}}= \sum_{j=0}^{\infty} 2^{j\theta q } \|\phi_j* f\|_{L_q}^q = \sum_{j=0}^{\infty} 2^{jd} \left\| 2^{j (k-\frac{d}{p})} \phi_j* f\right\|_{L_q}^q \stackrel{\eqref{eq: sampMarc}}{\geq} \frac{1}{ A^{q}}\sum_{j=0}^{\infty} \left\|\left(2^{j (k-\frac{d}{p})} \phi_j* f\right)\right\|^q_{\ell_q}.
\]
Similarly by \cite[ Volume II, Theorem 7$\cdot$ 5, ($7\cdot7$), pp.~28-29]{Zyg} we have estimate with a constant depending only on $d$ and $r$:
\[
\begin{split}
\sum_{j=0}^{\infty} \left\|\left(2^{j (k-\frac{d}{p})} \phi_j* f\right)\right\|^r_{\ell_r} &\gtrsim
\sum_{j=0}^{\infty} 2^{jd} \left\|\left(2^{j (k-\frac{d}{p})} \phi_j* f\right)\right\|^r_{L_r}
\\&=\sum_{j=0}^{\infty} 2^{j(d+kr- \frac{dr}{p})} \left\| \phi_j* f\right\|^r_{L_r}= \sum_{j=0}^{\infty} 2^{j\lambda r} \left\| \phi_j* f\right\|^r_{L_r}= \|f\|_{B^{\lambda}_{r,r}}.
\end{split}
\]

Then by the Bennet-Carl theorem (see \cite{Pie} 1.6.6)  $Id:\ell_q\rightarrow\ell_r$ is a $(v,1)$-summing operator, where $\frac{1}{v}=\frac{1}{q}-\frac{1}{r}+\frac{1}{2}$ and  
\[
\pi_{v,1}(Id:\ell_q\rightarrow\ell_r)\leq K,  
\] 
where $K$ is independent of $p$. Therefore there exists $C_1>$ independent of $p$ such that
\[
\pi_{v,1}(T_3)\leq C_4,  
\] 
By (\cite{BeLe}, Th. 6.4.4 and Th. 6.5.1)  we have
\[
\|T_4\|\leq C_5.
\]
Therefore $S_{d,k,p}$ is a $(v,1)$-summing operator and
\[
\pi_{v,1}(S_{d,k,p})\leq C_6 (q-p)^{1-\frac{2}{p}}.
\]
We rewrite above inequality in terms of $p_0=\max\{\frac{2d}{2k+d},p\}$ and $v$. Since $r= \min\{s,2\}$, by a routine calculation we get
\[
\pi_{v,1}(S_{d,k,p})\leq K (v-p_0)^{1-\frac{2}{p}}.
\]
for $v\in (p_0,2)$. Therefore
\[
\int_{p_0}^{2} (\pi_{v,1}(S_{d,k,p}))^{v} (v-p_0)^{\alpha} dv <\infty
\]
for $\alpha > p_0(1-\frac{2}{p})-1 $. Applying the Theorem \ref{extrapol} we get the desired result.
\end{proof}


\bibliographystyle{plain}
\bibliography{biblio}

\end{document}